\documentclass[11pt]{article}

\usepackage{amsmath,amsfonts,amsthm,amssymb} % Math packages
\usepackage{tikz}
\usepackage{lipsum} % Used for inserting dummy 'Lorem ipsum' text into the template
\usetikzlibrary{snakes}
\usepackage{graphicx}
\usepackage{xfrac}   
\usetikzlibrary{calc,3d}
\usetikzlibrary{arrows,backgrounds,snakes,shapes}	

\usepackage{hyperref}
\hypersetup{colorlinks=true,pdftitle="",pdftex}

\newtheorem{thm}{Theorem}[section]
\newtheorem{prop}[thm]{Proposition}
\newtheorem{cor}[thm]{Corollary}
\newtheorem{lem}[thm]{Lemma}

\theoremstyle{definition}
\newtheorem{defn}[thm]{Definition}

\theoremstyle{remark}

\def\be{\begin{eqnarray}}
\def\ee{\end{eqnarray}}
\def\ben{\begin{eqnarray*}}
\def\een{\end{eqnarray*}}

\numberwithin{equation}{section}

\newcommand{\zpz}{{\mathbb{Z}}\big/{p\mathbb{Z}}}

\begin{document}

\title{Majorization and R\'enyi Entropy Inequalities via Sperner Theory\footnote{Some portions of this paper were presented by the authors at the IEEE International Symposia on Information Theory 
in 2014 and 2015 \cite{WWM14:isit, WM15:isit}.}}

\author{
Mokshay Madiman\thanks{Department of Mathematical Sciences, University of Delaware, Newark DE 19716. Email: madiman@udel.edu},
Liyao Wang\thanks{J. P. Morgan Chase, New York. Email: njuwangliyao@gmail.com} 
and Jae Oh Woo\thanks{Department of Mathematics, and Department of Electrical and Computer Engineering, The University of Texas at Austin. Email: jaeoh.woo@aya.yale.edu}}

%%    Information for first author
%\author{Mokshay Madiman}
%%    Address of record for the research reported here
%\address{Department of Mathematics, University of Delaware}
%%    Current address
%%\curraddr{Department of Mathematics and Statistics,
%%Case Western Reserve University, Cleveland, Ohio 43403}
%\email{madiman@udel.edu}
%%    \thanks will become a 1st page footnote.
%%\thanks{The first author was supported in part by NSF Grant \#000000.}
%
%%    Information for second author
%\author{Liyao Wang}
%\address{JP Morgan Chase}
%\email{njuwangliyao@gmail.com}
%%\thanks{Support information for the second author.}
%
%\author{Jae Oh Woo}
%\address{Department of Mathematics, and Department of Electrical and Computer Engineering, The University of Texas at Austin}
%\email{jaeoh.woo@utexas.edu}
%%\thanks{Support information for the second author.}

%    General info
%\subjclass[2000]{Primary 54C40, 14E20; Secondary 46E25, 20C20}

%\date{January 1, 2001 and, in revised form, June 22, 2001.}

%\dedicatory{d}

%\keywords{Posets; Strong Sperner property; Majorization; Entropy Power Inequality; Sumset Inequality.}

\date{13 November 2018} %\today

\maketitle
\begin{abstract}
A natural link between the notions of majorization and strongly Sperner posets is elucidated.
It is then used to obtain a variety of consequences, including new R\'enyi entropy inequalities for 
sums of independent, integer-valued random variables.
\end{abstract}
%A ranked poset is called strongly Sperner if the size of $k$-family cannot exceed the sum of $k$-largest Whitney numbers. 
%In a sense of function ordering, a function $f$ is (weakly) majorized by $g$ if the the sum of $k$-largest values in $f$ cannot 
%exceed the sum of $k$-largest values in $g$. Two definitions arise from different contexts, but each share a strong similarity with each other. 
%Furthermore, the product of two weighted posets assumes a structural similarity with a convolution of two functions. 
%Elements in the product of weighted posets with ranks capture underlying structures of the building blocks in the convolution. 
%Combining all together, we are able to derive 

%\tableofcontents

%\section*{This is an unnumbered first-level section head}
%This is an example of an unnumbered first-level heading.

%% The correct journal style for \specialsection is all uppercase; a known bug
%% in amsart.cls prevents this, so input must be uppercase until it is fixed.
%\specialsection*{This is a Special Section Head}
%\specialsection*{THIS IS A SPECIAL SECTION HEAD}
%This is an example of a special section head%
%%%%%%%%%%%%%%%%%%%%%%%%%%%%%%%%%%%%%%%%%%%%%%%%%%%%%%%%%%%%%%%%%%%%%%%%
%\footnote{Here is an example of a footnote. Notice that this footnote
%text is running on so that it can stand as an example of how a footnote
%with separate paragraphs should be written.
%\par
%And here is the beginning of the second paragraph.}%
%%%%%%%%%%%%%%%%%%%%%%%%%%%%%%%%%%%%%%%%%%%%%%%%%%%%%%%%%%%%%%%%%%%%%%%%
%.
\section{Introduction}\label{sec:intro}

It was observed by Erd\H{o}s \cite{Erd45} in 1945 that the lemma of Littlewood and Offord~\cite{LO43} on small ball probabilities of weighted sums of  
Bernoulli random variables actually follows from Sperner's theorem \cite{Spe28} on the maximal size of antichains in the Boolean lattice. Subsequently Stanley \cite{Sta80} and Proctor \cite{Pro82}
used similar ideas to attack more difficult problems; a very nice review of the key ideas can be found in \cite{Kri16}. The goal of this paper is to further develop the core idea of relating properties of posets to
the ``distributional spread'' of weighted sums of independent, integer-valued random variables. We do this in two steps.
First, we elucidate a natural link, which does not seem to have been explicitly observed in the literature, between the strong Sperner property of posets and its behavior
for product posets on the one hand, and majorization inequalities on the other. Second, we follow a classical approach, similar to that used in our earlier
papers \cite{WM14,WWM14:isit,MWW17:1}, to  demonstrate
new R\'enyi entropy inequalities for sums of independent random variables using the majorization inequalities.
The entropy inequalities are of interest in information theory and probability, and were our original motivation for this work-- they are discussed
at length in Section~\ref{sec:epi-appln}.

%This paper is focused on building majorization and R\'enyi entropy inequalities leveraging Sperner Theory. We explicitly prove that a product of strong Sperner posets can be translated to an entropy inequality in a discrete domain, namely, a discrete entropy power inequality. As a crucial intermediate step, we establish the link between the notions of majorization and strong Sperner posets.

%\par\vspace{.1in}
%\noindent{\bf Our Contributions.}
%Our main contributions in this paper are the following. We first extend the antichain observation of Erd\H{o}s for Sperner posets 
%to strong Sperner posets. Next we establish the close similarity between the convolution of two functions and the 
%product of two posets. Finally, we prove that strong Sperner posets translate to majorization in the sense of function ordering. 
%Then by following our previous approaches in , we prove several types of R\'enyi entropy 
%inequalities including discrete entropy power inequalities, and discuss various applications.

%\par\vspace{.1in}
%\noindent{\bf Notation, Terminology, and Majorization Lemma.}
%Throughout this paper, three major notations are used to describe our main results. We also note that the background material related to a poset is described in Section \ref{section:poset}.

%\textbf{(N1) Function and Random Variable Representations.} 

In order to state our main results, we need to develop some terminology. 
%We use the notion of a $\#$-log-concave function. 
%The definition of functional $\#$-log-concavity is originally motivated by weighted posets with log-concave weighted Whitney numbers. 
For a non-negative function $f:\mathbb{Z}\to\mathbb{R}_+$ over the integers, the support $\text{Supp}(f)$ is defined by $\{x\in\mathbb{Z}: f(x)>0\}$.
We identify sets with their indicator functions; thus, for example, $f=0.4\{0,3\}+ 0.2\{2\}$ means $f(0)=f(3)=0.4$, $f(2)=0.2$,
and $f(x)=0$ for $x\in\mathbb{Z}\setminus\{0,2,3\}$.

%For the consistency with the rank of the (ranked) partially ordered set, the starting index is chosen to be $0$ instead of $1$. 
\begin{defn}\label{defn:hash}
Suppose $f:\mathbb{Z}\to\mathbb{R}_+$ is finitely supported, with $|\text{Supp}(f)|=n+1$. 
Then we may write $\text{Supp}(f)=\{x_0,\cdots,x_n\}$ with $x_0<\cdots<x_n$,
and we may represent $f$ in the form
\begin{align*}
f=\sum_{r=0}^{n} a_r \{x_r\},
\end{align*}
where $a_i> 0$ for each $i\in \{0,\ldots, n\}$. Given the non-negative function $f$, we define $f^{\#}$ by
\begin{align*}
f^{\#}:=\sum_{r=0}^{n} a_r \{r\}.
\end{align*}
\end{defn}

Thus, $f^{\#}$ is supported on $\{0,\cdots,n\}$ and it takes the same functional values as $f$. 
If we consider the graph, we may think of $f^{\#}$ as a ``squeezed rearrangement'' of $f$,
where we preserve the order of the function values but eliminate gaps in the support. 

\begin{defn}
We say $f$ is \textit{$\#$-log-concave} if $f^{\#}$ is log-concave, i.e., $f^{\#}(i)^2 \geq f^{\#}(i-1)f^{\#}(i+1)$ for any $i\in\mathbb{Z}$. 
%Similarly, the notion $\#$-log-concavity can be extended to random variables in a similar way. 
We say that a random variable $X$ taking values in the integers is \textit{$\#$-log-concave} if its probability mass function is $\#$-log-concave. 
Given a random variable $X$ with probability mass function $f$, we write $X^{\#}$ for a random variable with probability mass function $f^{\#}$. 
\end{defn}

In the terminology of Definition~\ref{defn:hash}, since $a_r=0$ for $r\in\mathbb{Z}\setminus\{0,\cdots,n\}$, $f$ is \textit{$\#$-log-concave} if and only if $a_r^2\geq a_{r-1}a_{r+1}$.

%\textbf{(N2) Majorization.} 
We also need the classical notion of majorization. We use $f^{[i]}$ to denote the $i$-th largest value of $f$,
allowing for the possibility of multiple ties. For example, $f^{[i]}=f^{[i+1]}$ when $i$-th large value appears at two different arguments. 

\begin{defn}
Consider two finitely supported functions $f$ and $g$ from $\mathbb{Z}$ to $\mathbb{R}_+$,
and assume $|\text{Supp}(f)|=|\text{Supp}(g)|=n+1$.
We say $f$ is \textit{majorized} by $g$ (and write $f \prec g$) if 
\begin{align}\label{eqn:majorization_condition_1}
\sum_{i=1}^{k} f^{[i]} \leq \sum_{i=1}^{k}g^{[i]} \quad \text{for all } k=1,\cdots,n,
\end{align}
and
\begin{align}\label{eqn:majorization_condition_2}
\sum_{i=1}^{n+1} f^{[i]} = \sum_{i=1}^{n+1} g^{[i]}.
\end{align}
%If $f$ and $g$ satisfy both conditions \eqref{eqn:majorization_condition_1} and \eqref{eqn:majorization_condition_2}, 
For random variables $X$ and $Y$ with probability mass functions $f$ and $g$ respectively, we write $X\prec Y$ if $f\prec g$.
\end{defn}

%\section{Majorization for Sums of Discrete Random Variables}

Our first main theorem is a majorization inequality for convolutions that holds under a log-concavity condition.
Recall that, given independent random variables $X, Y$ with probability mass functions $f, g$, 
the sum $X+Y$ has the probability mass function $f\star g$, where $\star$ denotes convolution, i.e.,
$f \star g (k) = \sum_{i\in\mathbb{Z}} f(i)g(k-i)$ for each $k\in\mathbb{Z}$.

\begin{thm}\label{thm:main1}
Let $N$ be a finite number. If $X_1,\cdots,X_N$ are independent and $\#$-log-concave over $\mathbb{Z}$, then
\begin{align}
X_1 + \cdots + X_N \prec X_1^{\#} + \cdots + X_N^{\#}.
\end{align}
\end{thm}

The proof of Theorem~\ref{thm:main1} is based on the strong Sperner property of the product of weighted chain posets-- 
Section~\ref{section:poset} summarizes the necessary background on poset theory, and the proof of the theorem is detailed
in Section~\ref{sec:thmproofmain1}. We mention in passing that although we focus on $\mathbb{Z}$-valued random variables
with finite support in this paper, Theorem~\ref{thm:main1} has an extension to the case where the random variables 
have infinite support using a similar procedure to that in \cite{WWM14:isit, Woo15:phd, MWW17:1}.

%\section{Miscellaneous Application to Set Majorization over the Integers}\label{sec:gen_rearrange}

Let us discuss a pleasing application of Theorem~\ref{thm:main1} to proving a key ingredient
in rearrangement inequalities on the integers proved by Gabriel \cite{Gab32} (generalizing
a result of Hardy and Littlewood \cite{HL28}) and popularized in the book by Hardy, Littlewood and P\'olya \cite{HLP88:book}.
For a finite set $A$ in $\mathbb{Z}$, note that $A^{\#}=\{0,1,\cdots,|A|-1\}$; 
here, as before, we identify the sets $A$ and $A^{\#}$ with their indicator functions.

\begin{cor}\label{cor:setmajorization}
If $A_1,A_2,\cdots,A_N$ are finite sets (or indicator functions) in $\mathbb{Z}$, then
\begin{align*}
{A_1} \star A_2 \star \cdots  \star A_N \prec A_1^{\#} \star A_2^{\#} \star \cdots  \star A_N^{\#}.
\end{align*}
\end{cor}

To see how Corollary~\ref{cor:setmajorization} follows from Theorem \ref{thm:main1}, suppose 
random variables $X_1,X_2,\cdots,X_N$ are uniformly distributed on the sets $A_1,A_2,\cdots,A_N$ respectively. 
By applying Theorem \ref{thm:main1} and re-normalizing the probability into the total number of points, the desired result follows.

We note that a version of Gabriel's inequality was, in fact, extended to  the prime cyclic groups $\zpz$
by Lev \cite{Lev01} . In our companion paper \cite{MWW17:1}, we develop a further generalization of such
 rearrangement inequalities (see \cite[Theorem 6.2]{MWW17:1}) in the prime cyclic groups, with the crucial
 in our proofs being the leveraging  of Lev's set majorization lemma (see \cite[Theorem 1]{Lev01} for the full statement). 
% While Lev's result  is more general in that it applies to all prime cyclic groups,
%we give below a quick proof of the set majorization lemma over the integers as an application of Theorem \ref{thm:main1}. 
The results of Gabriel \cite{Gab32}, Lev \cite{Lev01} and the authors \cite{MWW17:1} for general non-negative functions
rather than indicator functions of sets require additional assumptions because one has to take into account the ``shape''
of the convolved functions. It is a nice feature of the statement and proof above that it does not require such assumptions.

%SECOND THEOREM

Our second theorem is related to a beautiful and well known result of S\'ark\H{o}zy and Szemer\'edi \cite{SS65}
related to what they called the Erd\H{o}s-Moser problem (although the paper of Katona \cite{Kat66}, which they cite
a pre-publication version of, does not discuss it in the published paper, and the problem posed by Erd\H{o}s in 1947 
in the {\it American Mathematical Monthly} with solutions given by Moser \cite{EM47} as well as several others, 
which is cited by several later papers on the S\'ark\H{o}zy-Szemer\'edi result, seems only tangentially related).
%\subsection{Extension to Erd\H{o}s-Moser Problem}
In any case, the ``Erd\H{o}s-Moser problem'' is the following: Given $N$ i.i.d. Bernoulli random variables $Y_1,\cdots,Y_N$, 
estimate the maximal probability of independent weighted sums over distinct weights:
\begin{align*}
\sup_{k\in\mathbb{Z}} \quad\sup_{0<a_1\neq \cdots \neq a_N}\mathbf{P}\left( a_1Y_1 + a_2Y_2+\cdots + a_NY_N = k\right).
\end{align*}
S\'ark\H{o}zy and Szemer\'edi \cite{SS65} asserted that Erd\H{o}s and Moser had shown that the maximal probability is of order $\big(\frac{\log N}{N}\big)^{3/2}$
and had conjectured that the logarithmic term could be removed; they proved this conjecture, thus showing that the  maximal probability is of order $N^{-3/2}$.
However, identification of an extremal set of weights remained open until Stanley \cite{Sta80} used tools from algebraic geometry
to show that $(a_1, a_2, \ldots, a_N)=(1, 2, \ldots, N)$ is extremal. A more elementary algebraic proof was soon after given
by Proctor \cite{Pro82}. Much more recently, Nguyen \cite{Ngu12} not only observed that the maximal probability is in fact
$[\sqrt{24/\pi} +o(1)]N^{-3/2}$, but he also showed a stability result around the extremal configuration.

With this background, we are ready to state our second main result.

\begin{thm}\label{thm:main2}
	Assume that $0<a_1< a_2 < \cdots < a_N$. If $Y_i$'s are independent random variables following 
	$\text{Binomial}\left( m_i,\frac{1}{2} \right)$ for $1\leq m_N \leq m_{N-1}\leq \cdots \leq m_1$, then
	\begin{align}
	a_1Y_1 + a_2Y_2+\cdots + a_NY_N \prec Y_1 + 2Y_2+ \cdots + NY_N.
	\end{align}
\end{thm}

The proof of Theorem~\ref{thm:main2} is based on the strong Sperner property of some products of posets, and is detailed in Section~\ref{sec:thmproofmain2}. 

As a direct application of Theorem \ref{thm:main2}, %or Proposition \ref{thm:erdosmoser1}, 
we can go beyond the Bernoulli assumption in the prior studies of \cite{SS65, Sta80} discussed above,
and identify the extremal weights for the wider class of binomially distributed random variables. 

\begin{cor}\label{cor:erdosmoser}
	Let $0<a_1\neq \cdots \neq a_N$. Let $Y_i$ be i.i.d. random variables with the $\text{Binomial}\left( m,\frac{1}{2} \right)$ distribution. 
	Then
	%For the case when $\alpha=+\infty$ in Proposition \ref{thm:erdosmoser1}, we have
\begin{align*}
\mathbf{P}\left( a_1Y_1 + a_2Y_2+\cdots + a_NY_N = k\right) \leq \mathbf{P}\left( Y_1 + 2Y_2+ \cdots + NY_N = \left\lfloor\frac{mN(N+1)}{4} \right\rfloor \right).
\end{align*}
\end{cor}

%\begin{proof}
To see how Corollary~\ref{cor:erdosmoser} follows from Theorem \ref{thm:main2}, observe that since $Y_i$ are i.i.d.
in the former, we may assume that the $a_i$ are ordered. Then the conclusion follows from Theorem \ref{thm:main2}, 
and the optimal case is achieved at the midpoint of the range as the distribution of $ Y_1 + 2Y_2+ \cdots + NY_N$ is symmetric 
and unimodal (this latter fact is confirmed by Lemma \ref{lem:productmnpeck}, which we discuss later).
%\end{proof}
Of course, Theorem \ref{thm:main2} can be applied without an 
identically distributed assumption, but we make this assumption in Corollary~\ref{cor:erdosmoser} for simplicity of statement.

This paper is organized as follows. Our original motivation for pursuing this work came from a search for R\'enyi entropy power inequalities
for integer-valued random variables, which is a problem of significant interest in information theory. We explain this motivation
and describe how our main results may be applied to obtain new entropy power inequalities in Section~\ref{sec:epi-appln}.
The rest of the paper focuses on the proofs of our main results-- Section~\ref{section:poset} recalls the necessary background on Sperner theory, 
and the proofs of the two main theorems are detailed in Sections~\ref{sec:thmproofmain1} and ~\ref{sec:thmproofmain2}.

\section{Applications to R\'enyi Entropy Inequalities}\label{sec:epi-appln}

\subsection{Background on entropy power inequality}\label{sec:entropyin-bg}

%\textbf{(N3) R\'enyi Entropy.} 

We first define the one-parameter family of R\'enyi entropies for a probability mass function on the integers
(these can be defined on more general spaces by using a reference measure other than counting measure,
but we do not need the more general notion here).

\begin{defn}
Let $X$ be an integer-valued random variable with probability mass function $f$.
The R\'enyi entropy of order $\alpha\in(0,1) \cup (1,+\infty)$ is defined by
\begin{align*}
H_\alpha  (X)= \frac{1}{1-\alpha } \log\left( \sum_{i\in \mathbb{Z}} f(i)^{\alpha } \right).
\end{align*}
For limiting cases of $\alpha$, define
\begin{align*}
H_0(X) &= \log|\text{supp}(f)|,\\
H_1 (X) &= \sum_{i\in \mathbb{Z}} -f(i)\log f(i),\\
H_\infty (X) &= -\log \sup_{i\in \mathbb{Z}} f(i).
\end{align*}
\end{defn}

The three special cases are defined in a manner consistent with taking limits of $H_\alpha(X)$ for $\alpha\in(0,1) \cup (1,+\infty)$. 
Thus the R\'enyi entropy of order $\alpha \in [0,\infty]$ is well-defined. In particular, $H_1(\cdot)$ is simply the Shannon entropy $H(\cdot)$. 

Entropy inequalities (even just for Shannon entropy) are powerful tools that have found use in virtually all 
parts of mathematics. For example, just within discrete mathematics, they have been used to obtain bounds for enumeration problems
 (see, e.g., \cite{Rad01, MT10, JKM13}), to prove sumset inequalities in additive combinatorics (see, e.g., \cite{Ruz09:1, Tao10, MMT12, ALM17, MK18}),
 and to study probabilistic models of discrete phenomena (e.g., independent sets \cite{Kah01b}, card shuffles \cite{Mor13}), colorings \cite{PS18}).
% and to quantify limit theorems in discrete probability (see, e.g., \cite{HJK10}).
%
Among various entropy inequalities, the so-called ``entropy power inequality'' in Euclidean spaces has been very successfully 
applied to prove coding theorems or to determine channel capacities for communication problems involving 
Gaussian noise in Information Theory (see, e.g., \cite{Ber74, WSS06}). 
The entropy power inequality also plays an important role in Probability Theory (see, e.g., \cite{Joh04:book}) 
and Convex Geometry (see, e.g., \cite{DCT91, Gar02, MMX17:0}). 

The entropy power inequality can be formulated in two different ways. 
Firstly, the original formulation, which was suggested by Shannon \cite{Sha48} and proved by Stam \cite{Sta59}, 
is the following. For independent random variables $X$ and $Y$ in $\mathbb{R}^d$,
\begin{align*}
e^{\frac{2}{d}h(X+Y)} \geq e^{\frac{2}{d}h(X)} + e^{\frac{2}{d}h(Y)},
\end{align*}
where $h(X)$ represents the  \textit{differential entropy} of $X$. Formally, if $X$ has a density function $f$ in $\mathbb{R}^d$, 
then $h(X)=-\int_{\mathbb{R}^d} f(x)\log f(x)dx$. The inequality shows the superadditivity of the ``entropy power'' with respect 
to the sum of two independent random variables. 
Secondly, an equivalent sharp formulation (see, e.g., \cite{WM14} for discussion) states that
\begin{align*}
h(X+Y) \geq h(Z_X^* + Z_Y^*),
\end{align*}
where $Z_X^*$ and $Z_Y^*$ are two independent Gaussian distributions with $h(X)=h(Z_X^*)$ and $h(Y)=h(Z_Y^*)$. 
%We note that $Z_X^*$ and $Z_Y^*$ can be obtained by shuffling domains of those density functions of $X$ and $Y$.

The entropy power inequality stated above focuses on the continuous setting of $\mathbb{R}^d$. 
It has been extensively studied and many refinements exist (see, e.g., \cite{ABBN04:1, MB07, MG18}). 
On the other hand, we only have a limited understanding of analogues of the entropy power inequality on 
discrete domains such as the integers or cyclic groups. 

One of the main difficulties is that useful analytic tools in the continuous domain cannot be naturally translated into the discrete domain. 
For example, one can derive the entropy power inequality in $\mathbb{R}^d$ from the sharp form of Young's inequality for convolution
developed by Beckner \cite{Bec75}, as observed by Lieb \cite{Lie78}, %\cite{Tos15:1, Li18}
or using inequalities for Fisher information, which is defined using derivatives of the probability density function, as done by Stam \cite{Sta59}. %\cite{MB07}. 
%Specifically, \textit{the sharp Young's convolution inequality} is described as follows. For $p,q,r>1$ such that $\frac{1}{p}+\frac{1}{q}=1+\frac{1}{r}$, 
%\begin{align*}
%\|f\star g\|_r \leq C_{p,q}\|f\|_p \|g\|_q,
%\end{align*}
%where $C_{p,q}<1$ and the optimal case is achieved when $f$ and $g$ are multi-dimensional Gaussian functions. Then by taking a derivative of the functional norm, one can deduce it to an entropy inequality.
%
%In addition, \textit{Fisher Information} of a continuous random variable $X$ with density $f$ is defined by $I(X)=\int_{\mathbb{R}^d} \left(\nabla \log(f(x))\right)^2 f(x) dx$. The process to derive the entropy power inequality from Fisher Information is more involved. Please refer to the paper \cite{MB07} for details. %These two well-known approaches are used to derive the entropy power inequality in a continuous domain.
Unfortunately, a non-trivial sharp Young's inequality cannot be achieved in the discrete setting. It is also not
obvious what the right definition of Fisher information should be for the discrete setting because discrete derivatives
do not satisfy the chain rule (see, e.g., \cite{KHJ05, Mad05:phd, BJKM10} for possible definitions of discrete Fisher informations). 
Owing to the difficulty of fitting such approaches into a discrete setting, a general and sharp analogue of the 
entropy power inequality on the integers has not yet been established.

Nevertheless, it is a natural and interesting question to find a fully satisfactory entropy power inequality on the integers-- earlier attempts
in this direction include \cite{JY10, HAT14, WWM14:isit, WM15:isit}. 
%As described in the second formulation of the entropy power inequality, finding entropy inequalities of sums is our main application focus. 
In our companion paper \cite{MWW17:1}, we established 
a lower bound on the entropy of sums in prime cyclic groups (including the integers) based on rearrangement inequalities and 
functional ordering by majorization where the rearrangement of a function $f$ is achieved by shuffling (permuting) the domain of $f$. 
While details may be found in \cite{MWW17:1}, the goal of these rearrangement inequalities is to identify optimal permutations that maximize
or minimize a sum of pairwise products. 

In this paper, we focus on the integer domain or the integer lattice domain. We continue to leverage the idea of majorization used 
in our companion paper \cite{MWW17:1}. However, instead of establishing rearrangement inequalities, we take a different path 
to establish the lower bound inequality of the entropy of sums in integers. Our approach is to establish and utilize the similarity
 between the strong Sperner property of posets (the origin of this
notion lies, of course, in Sperner's theorem \cite{Spe28}, but the way we use this notion is inspired by Erd\H{o}s~\cite{Erd45}
as described in Section~\ref{sec:intro}) 
and functional ordering by majorization. 
%The Sperner property of a poset and the antichain argument are originated from Erd\H{o}s~\cite{Erd45}.

\subsection{Two entropy inequalities}\label{sec:entropyin}

A key application of Theorem \ref{thm:main1} (and also Theorem \ref{thm:main2}) lies in establishing a lower bound on
the R\'enyi entropy of convolutions. As the main tool in translating majorization results to entropy inequalities, 
we use the following basic lemma.

%\textbf{Majorization Lemma.} 
\begin{lem}\cite[Proposition 3-C.1]{MOA11:book}\label{lem:maj_to_convex}
Assume that $f$ and $g$ are finitely supported non-negative functions in $\mathbb{Z}$ and $f\prec g$. For any convex function $\Phi:\mathbb{R}\to\mathbb{R}$,
\begin{align*}
\sum_{i\in\mathbb{Z}}\Phi\left( f(i) \right) \leq \sum_{i\in\mathbb{Z}}\Phi\left( g(i) \right).
\end{align*}
\end{lem}

We note that by choosing a convex function $\Phi(x) = -x^\alpha$ for $\alpha\in (0,1)$, $\Phi(x)=x\log x$ for $\alpha=1$, and $\Phi(x)=x^\alpha$ for $\alpha\in (1,+\infty)$, 
and by taking limits when $\alpha\in \{0,\infty\}$, 
inequalities for R\'enyi entropies of all orders follow from Lemma \ref{lem:maj_to_convex} whenever we have a probability mass function majorized by another. 
%In order to do so, decorating terms for R\'enyi entropy by multiplying $-1$ for $\alpha<1$, then taking $\log(\cdot)$ and multiplying $\frac{1}{1-\alpha}$.
In particular, our two main theorems combined with Lemma \ref{lem:maj_to_convex}  yield the following propositions.

%Then Lemma \ref{lem:maj_to_convex} implies the following two different types of entropy inequalities. 
%We omit those proofs since they are direct consequences of Lemma \ref{lem:maj_to_convex} by choosing convex functions $\Phi(x)=x^{\alpha}$ for $\alpha\in(1,+\infty)$ and $\Phi(x)=-x^{\alpha}$ for $\alpha\in(0,1)$. Then for $\alpha\in\{0,1,+\infty\}$, we take limits of $\alpha$ as described in (N3) above.

\begin{prop}\label{thm:hashconcave}
If $X_1,\cdots,X_N$ are independent and $\#$-log-concave over $\mathbb{Z}$, then
\begin{align}
H_{\alpha}\left( X_1 + \cdots + X_N \right) \geq H_{\alpha} \left(  X_1^{\#} + \cdots + X_N^{\#} \right),
\end{align}
for $\alpha \in [0,\infty]$.
\end{prop}

R\'enyi entropy inequalities such as this (and others of similar form in our companion paper \cite{MWW17:1}) 
have already begun finding utility (see, e.g., \cite{MMX17:1, XMM18}).

\begin{prop}\label{thm:erdosmoser1}
	Let $0<a_1< \cdots < a_N$. If $Y_i$'s are independent random variables following $\text{Binomial}\left( m_i,\frac{1}{2} \right)$ for $1\leq m_N \leq m_{N-1}\leq \cdots \leq m_1$, then
	\begin{align}
	H_{\alpha}\left( a_1Y_1 + a_2Y_2+\cdots + a_NY_N \right) \geq H_{\alpha} \left( Y_1 + 2Y_2+ \cdots + NY_N \right),
	\end{align}
	for $\alpha \in [0,\infty]$.
\end{prop}

Nguyen \cite{Ngu12} observed that the optimal solution of the Erd\H{o}s-Moser problem (for Bernoulli $Y_i$) minimizes the 
variance of $a_1Y_1 + a_2Y_2+\cdots + a_NY_N$ among all choices of distinct positive weights, i.e.,
for $0<a_1\neq \cdots \neq a_N$,
\begin{align*}
\text{Var}\left(Y_1 + 2Y_2+\cdots + NY_N\right)\leq \text{Var}\left(a_1Y_1 + a_2Y_2+\cdots + a_NY_N\right).
\end{align*}
Proposition \ref{thm:erdosmoser1} implies that the optimal solution of the Erd\H{o}s-Moser problem 
also minimizes the R\'enyi entropy for any order $\alpha\in[0,+\infty]$ (and even for the more general
binomial setting). 
%This observation seems interesting: for example,
%by taking $\alpha=2$ and finite sets $A_1, \ldots, A_N\subset \mathbb{Z}$ with , it implies that the number of solutions to the linear equation 
%\begin{align}\label{eq:collision}
%a_1 x_1 + \cdots + a_nx_n = a_1 x_1' + \cdots + a_nx_n' ,
%\end{align}
%where $x_i, x_i'\in A_i$ are equation variables, and the positive integer weights $a_i$ are distinct,
%is maximized when $a_1=\cdots=a_n=1$ and when $A_i=A_i^{\delta(A_i)}$.

\subsection{Inequalities for uniform distributions on subsets of $\mathbb{Z}^d$}

%The importance of an entropy power inequality is emphasized in Section \ref{sec:intro}. 
In \cite{WM15:isit}, we proved a discrete entropy power inequality for uniform distributions over finite subsets of the integers $\mathbb{Z}$. 
In the following lemma, we extend  \cite[Theorem II.2]{WM15:isit} from the $\alpha=1$ case to any R\'enyi entropy of order $\alpha\geq 1$. 
%The detailed proof of the lemma can be found in Section \ref{sec:pflem}.

\begin{lem}\label{lem:epi_uniforms}
If $X$ and $Y$ are independent and uniformly distributed over finite sets $A\subset \mathbb{Z}$ and $B\subset \mathbb{Z}$ respectively,
\begin{align}\label{eqn:EPI_uniform}
\mathcal{N}_\alpha(X+Y) + 1 \geq \mathcal{N}_\alpha(X) + \mathcal{N}_\alpha(Y),
\end{align}
where $\mathcal{N}_\alpha(X)=e^{(1+\alpha)H_\alpha(X)}$ for $\alpha\geq 1$.
\end{lem}

%\section{Proof of Lemma \ref{lem:epi_uniforms}}\label{sec:pflem}
\begin{proof}
Since the $\alpha=1$ case is proved in \cite{WM15:isit}, we assume that $\alpha>1$. Since any uniform distribution over a finite set is $\#$-log-concave, Theorem \ref{thm:hashconcave} implies that
	\begin{align*}
	H_\alpha(X+Y) \geq H_\alpha(X^\# + Y^\#).
	\end{align*}
Sincet $\mathcal{N}_\alpha(X)=\mathcal{N}_\alpha(X^\#)$ and $N_\alpha(Y)=N_\alpha(Y^\#)$ hold trivially, it suffices for proving the inequality \eqref{eqn:EPI_uniform} 
to only consider the case where $A$ and $B$ are sets of consecutive integers. Indeed, if we proved this special case,  we would have
	\begin{align*}
	\mathcal{N}_\alpha(X+Y)+1 &\geq \mathcal{N}_\alpha(X^\# + Y^\#) +1 \\
	&\geq \mathcal{N}_\alpha(X^\#) + \mathcal{N}_\alpha(Y^\#) \\
	&= \mathcal{N}_\alpha(X) + \mathcal{N}_\alpha(Y).
	\end{align*} 
which is the desired statement.
	
What remains is to prove \eqref{eqn:EPI_uniform} for uniform distributions on finite sets of consecutive integers. 
Let $|A|=n$ and $|B|=m$. Since the roles of $X$ and $Y$ are symmetric, we may assume that $n\geq m$. 
While $H_\alpha(X)= \log n$ and $H_\alpha(Y) = \log m$ due to uniformity, a direct calculation easily shows that
	\begin{equation}\label{eq:halph}
	H_\alpha(X+Y) = \frac{1}{1-\alpha}\log\left[ \sum_{i=1}^{m-1}\frac{i^\alpha}{m^\alpha n^\alpha} + (n-m+1)\frac{1}{n^{\alpha}} \right]. 
	\end{equation}
First, observe that if $m=1$, then $H_\alpha(X+Y)=H_\alpha(X)$ and $H_\alpha(Y)=0$. Then
$\mathcal{N}_\alpha(X+Y) + 1 = \mathcal{N}_\alpha(X) + 1 = \mathcal{N}_\alpha(X) + \mathcal{N}_\alpha(Y)$,
and hence the inequality (\ref{eqn:EPI_uniform}) is sharp. 
Next, consider the expression inside the logarithm in the formula \eqref{eq:halph}:
\begin{align*}
\sum_{i=1}^{m-1}\frac{i^\alpha}{m^\alpha n^\alpha} + (n-m+1)\frac{1}{n^{\alpha}}&= n^{-\alpha}m^{-\alpha}\sum_{i=1}^{m-1}i^\alpha+ n^{1-\alpha}-mn^{-\alpha}+n^{-\alpha}\\
&\leq n^{-\alpha}m^{-\alpha} \int_{1}^{m}x^\alpha dx + n^{1-\alpha}-mn^{-\alpha}+n^{-\alpha}\\
&=n^{1-\alpha} \left[ 1+\frac{1}{n}-\frac{\alpha}{1+\alpha}\frac{m}{n} -\frac{1}{1+\alpha}\frac{m^{-\alpha}}{n} \right].
\end{align*}
Plugging this bound into \eqref{eq:halph}, setting $k=m/n \in [0,1]$ and writing
	\begin{align*}
	\xi(k,n) = 1 + \frac{1}{n} -\frac{\alpha}{1+\alpha}k - \frac{1}{1+\alpha}k^{-\alpha}n^{-(1+\alpha)} ,
	\end{align*}
we obtain the following lower bound for the R\'enyi entropy of $X+Y$:
\begin{align*}
H_\alpha(X+Y) \geq \frac{1}{1-\alpha} \left[ \log n^{1-\alpha} + \log \xi(k,n) \right].
\end{align*}
Then
\begin{align*}
\mathcal{N}_{\alpha}(X+Y) = e^{(1+\alpha)H_{\alpha}(X+Y)}\geq e^{(1+\alpha)\log n} e^{\frac{1+\alpha}{1-\alpha}\log\xi(k,n)}=:\nu(k,n).
\end{align*}
Since the $m=1$ case is already proved, if the following inequality is true, we are done.
\begin{align*}
\nu(k,n)=e^{(1+\alpha)\log n} e^{\frac{1+\alpha}{1-\alpha}\log\xi(k,n)} \geq e^{(1+\alpha)\log n} + e^{(1+\alpha)\log kn} = \mathcal{N}_{\alpha}(X)+\mathcal{N}_{\alpha}(Y).
\end{align*}
By rearranging terms, the above inequality is equivalent to
\begin{align*}
\xi(k,n) \leq \left( 1+k^{1+\alpha} \right)^{\frac{1-\alpha}{1+\alpha}}.
\end{align*}
When $k=\frac{2}{n}$, we can directly show the following inequality is true by elementary calculation:
\begin{align*}
\xi\left(\frac{2}{n},n \right)=1+\frac{1-\alpha}{1+\alpha}\frac{1}{n}-\frac{2^{-\alpha}}{1+\alpha}\frac{1}{n} \leq \left( 1+2^{1+\alpha}n^{-(1+\alpha)} \right)^{\frac{1-\alpha}{1+\alpha}}.
\end{align*}
Furthermore, it is easy to show that for $k\in[0,1]$ and $\alpha>1$
\begin{align*}
\left( 1+k^{1+\alpha} \right)^{\frac{1-\alpha}{1+\alpha}} \geq 1-\left(1-2^{\frac{1-\alpha}{1+\alpha}}\right)k.
\end{align*}
Hence it suffices to show 
\begin{align*}
\xi(k, n) \leq  1- \left(1-2^{\frac{1-\alpha}{1+\alpha}}\right)k
\end{align*}
for all $n\geq 2$ and $k\geq \frac{3}{n}$. Let $\phi(k,n):=1- \left(1-2^{\frac{1-\alpha}{1+\alpha}}\right)k - \xi(k,n)$. For a fixed $n\geq 2$,
\begin{align*}
\frac{\partial \phi }{\partial k} =-1 +2^{\frac{1-\alpha}{1+\alpha}}+\frac{\alpha}{1+\alpha}\left( 1-k^{-(1+\alpha)}n^{-(1+\alpha)}\right).
\end{align*}
Then $\frac{\partial \phi }{\partial k}=0$ at
\begin{align*}
k^*=\frac{1}{n} \alpha^{\frac{1}{1+\alpha}}\left[ -1 + 2^{\frac{1-\alpha}{1+\alpha}}+2^{\frac{1-\alpha}{1+\alpha}}\alpha \right]^{-\frac{1}{1+\alpha}},
\end{align*}
and $\frac{\partial \phi }{\partial k}>0$ for $k>k^*$ and $\frac{\partial \phi }{\partial k}<0$ for $k<k^*$. Finally, by elementary calculation, we can easily show that 
\begin{align*}
\frac{1}{n} \leq k^* < \frac{2}{n}.
\end{align*}
Thus, $\phi(k,n)$ is minimized at $k=\frac{3}{n}$ for $k\geq \frac{3}{n}$. By elementary calculation, we can also confirm that $\phi\left(\frac{3}{n},n\right)\geq 0$,
which completes the proof.

We note that when $k=\frac{2}{n}$, $\phi\left(\frac{2}{n},n\right) <0$ for some $n$;
this is why $\phi(k,n)\geq 0$ is only proved when $k\geq \frac{3}{n}$ rather than $k\geq \frac{2}{n}$. 
\end{proof}

We remark that $+1$ term on the left side of inequality \eqref{eqn:EPI_uniform} is only necessary for one-point mass distributions.
In other words, if If $X$ and $Y$ are independent and uniformly distributed over finite sets of cardinality at least 2, then
we in fact have 
\begin{align}
\mathcal{N}_\alpha(X+Y)  \geq \mathcal{N}_\alpha(X) + \mathcal{N}_\alpha(Y).
\end{align}
However, we highlight the formulation with $+1$ both because of the similarity with the Cauchy-Davenport Theorem \cite{Ruz09:2},
and because the discrete entropy power inequality over the integers in Lemma~\ref{lem:epi_uniforms} can be extended to the integer lattice $\mathbb{Z}^d$.

\begin{thm}\label{thm:discrete-epi}
If $X$ and $Y$ are uniform distributions over finite sets $A$ and $B$ in $\mathbb{Z}^d$,
\begin{align*}
\mathcal{N}_\alpha(X+Y) +1 \geq \mathcal{N}_\alpha(X)+\mathcal{N}_\alpha(Y),
\end{align*}
where $\mathcal{N}_\alpha(X)=e^{(1+\alpha)H_\alpha(X)}$ for $\alpha\geq 1$.
\end{thm}
\begin{proof}
%We reduce the problem in $\mathbb{Z}^d$ to the same problem in $\mathbb{Z}$. Then by applying Lemma~\ref{lem:epi_uniforms}, the proof follows. 
%
Consider a point $\mathbf{z}=(z_1,\cdots,z_d)$ in $\mathbb{Z}^d$ where $z_i\geq 0$ for each $i$. 
%If $q$ is not big enough, we can choose arbitrarily large $q$ satisfying the desired conditions explained below. 
We regard $\mathbf{z}$ as a $q$-ary representation of an integer value, where $q$ is large and chosen later. 
In other words, the point $\mathbf{z}\in \mathbb{Z}^d$ can be mapped to a unique integer value in $\mathbb{Z}$:
\begin{align}\label{eq:q-ary}
\mathbf{z}=(z_1,\cdots,z_d) \mapsto z_1q^{d-1}+z_2q^{d-2}+\cdots+z_d \in \mathbb{Z}.
\end{align}

For the set $A$ in $\mathbb{Z}^d$, without loss of generality, we can shift $A$ so that each point contains only non-negative components. 
Let $A'$ be the set in $\mathbb{Z}$ equivalent to $A$ via the $q$-ary representation \eqref{eq:q-ary}. 
Similarly we can find the set $B'$ in $\mathbb{Z}$ equivalent to $B$ in $\mathbb{Z}^d$. 
We choose $q$ large enough so that $A\star B$ in $\mathbb{Z}^d$ maps to $A' \star B'$ in $\mathbb{Z}$ via the $q$-ary representation.
 
Let $X'$ and $Y'$ be uniform distributions on $A'$ and $B'$ in $\mathbb{Z}$, respectively. This implies
\begin{align*}
H_\alpha(X+Y) = H_\alpha(X'+Y') \geq H_\alpha(X'^{\#}+Y'^{\#}).
\end{align*}
Then the conclusion follows by applying Lemma~\ref{lem:epi_uniforms} and from the fact that $H_\alpha(X)=H_\alpha(X')$ and $H_\alpha(Y)=H_\alpha(Y')$.
\end{proof}

Theorem~\ref{thm:discrete-epi} uses the exponent $c=1+\alpha$; R\'enyi entropy power inequalities with the same exponent in $\mathbb{R}$
were recently explored by Bobkov and Marsiglietti \cite{BM17} (although it was shown soon after by Li \cite{Li18} that this exponent can be improved).
In fact, these authors proved similar inequalities in $\mathbb{R}^d$, with the exponent $\frac{1+\alpha}{d}$, mimicking the $2/d$ exponent in
the original Shannon-Stam entropy power inequality.

The inequality we derived in Theorem \ref{thm:discrete-epi} is independent of the dimensional factor $d$, 
and one might wonder whether a entropy power inequality in the integer lattice $\mathbb{Z}^d$
that respects the dimension exists. Even for the subclass of uniform distributions and for $\alpha=1$, however, one can
easily construct counterexamples  showing that an exponent of $2/d$ fails in general in $\mathbb{Z}^d$.
% dependent of the dimensional factor $d$ under a certain dimensional assumption, for example,
%\begin{align}\label{eq:ddimepi}
%\mathcal{N}_{\alpha}^d(X+Y)  +1 \geq \mathcal{N}_{\alpha}^d(X)+\mathcal{N}_{\alpha}^d(Y),
%\end{align}
%where $\mathcal{N}_{\alpha}^d(X)=e^{\frac{1+\alpha}{d}H(X)}$. A proper dimensional assumption of $X$ and $Y$ is still open to the authors to achieve the inequality \eqref{eq:ddimepi}. 
We remark that in order to develop a discrete Brunn-Minkowski inequality in the integer lattice,  Gardner and Gronchi~\cite{GG01}
imposed a natural and appropriate dimensional assumption, the main point of which is that at least two points
should be assigned to each axis direction. However, the dimensional assumption from \cite{GG01} is still not sufficient to 
obtain an improvement of  Theorem \ref{thm:discrete-epi} with exponent $\frac{1+\alpha}{d}$ (as can be checked by counterexamples).
Hence we leave the discovery of appropriate dimensional entropy inequalities in the integer lattice as an open question for future works.

\section{Background on Sperner Theory}\label{section:poset}

In this section, we summarize the basic elements of Sperner Theory as needed for our proofs. 
A comprehensive summary can be found in books by Stanley \cite{Sta12:book} and Engel \cite{Eng97:book}.

\subsection{Partially ordered set (poset)}
A set $S$ with a binary relation $\preccurlyeq$ is said to be \textit{partially ordered} if the relation $\preccurlyeq$ satisfies the 
reflexive, anti-symmetric, and transitive properties, i.e., for any $a,b,c\in S$,
\begin{itemize}
	\item {$a\preccurlyeq a$ (reflexive),}
	\item {if $a\preccurlyeq b$ and $b\preccurlyeq a$, then $a=a$ (antisymmetric),}
	\item {if $a\preccurlyeq b$ and $b\preccurlyeq c$, then $a\preccurlyeq c$ (transitive).}
\end{itemize}
We emphasize that we use the symbol  $\prec$ to represent majorization, and this has no relation to the partial order $\preceq$. 
If $S$ is partially ordered, we call $S$ a \textit{partially ordered set}, or a \textit{poset}. 
For $a,b\in S$, $a$ and $b$ are \textit{comparable} if $a\preccurlyeq b$ or $b\preccurlyeq a$. 
Otherwise, $a$ and $b$ are \textit{incomparable}. A \textit{chain poset} is a poset in which any two elements are comparable. 
A subset $C$ of $S$ is called a \textit{chain} of $S$ if $C$ is a chain poset as a sub-poset of $S$. 
We define \textit{the length} of a chain $C$ to be the number of elements in $C$.

A subset $A$ of $S$ is called an \textit{antichain} if any two distinct elements of $A$ are incomparable. A subset $K$ of $S$ is called a \textit{$k$-family} of $S$ if it is a union of at most $k$ antichains.
We say a poset $S$ is \textit{weighted} if each element has a positive weight. 
The \textit{weight function} $w:S\to\mathbb{R}_+$ defines the weight of each element in $S$. We use a triple $\left(S,w,\preccurlyeq \right)$ to represent the weighted poset, but we sometimes omit to mention the weight function $w$ explicitly when we describe a weighted poset. If the poset has no weight function (or \textit{unweighted}), we implicitly assume that each weight of an element is $1$.

A chain $C$ of $S$ is called \textit{maximal} if there is no larger chain $C'$ such that $C\subseteq C'$. An element $s$ in $S$ is called \textit{minimal} if $t\preceq s$ implies $s=t$ in $S$. The element $s$ of $S$ is said to \textit{cover} the element $t$ of $S$ if $t\preceq s$ and if $t\preceq s'\preceq s$ implies $s=s'$ when $s'\neq t$. If every maximal chain of the poset $S$ has length $n+1$, we call $S$ a \textit{graded poset} with rank $n$. In such a case, we can define a unique \textit{rank function} $\rho:S\to\{0,1,\cdots, n\}$ of $S$ such that $\rho(a)=0$ if $a$ is a minimal element of $S$, and $\rho(b)=\rho(a)+1$ if $b$ covers $a$. Then the \textit{rank of $a$} is assigned to be $\rho(a)$. Given a weighted and ranked poset $\left(S,\preceq,w\right)$, the sum of all weights at the same rank $r\in\{0,1,\cdots, n\}$ is called the \textit{weighted Whitney number of the rank $r$}. Similarly, if the poset is unweighted, the \textit{Whitney number of the rank $r$} is the total number of elements at the rank $i$. We say that weighted Whitney numbers are \textit{log-concave} if the sequence of weighted Whitney numbers is log-concave in an increasing order of the rank. We say that weighted Whitney numbers are \textit{rank-symmetric} if the sequence of weighted Whitney numbers is symmetric in an increasing order of the rank. Similarly, we say that weighted Whitney numbers are \textit{rank-unimodal} if the sequence of weighted Whitney numbers is unimodal in an increasing order of the rank.

The weighted and ranked poset $\left(S,\preceq,w\right)$ has \textit{$k$-Sperner property} if the maximum total weight among all $k$-families in $S$ equals the largest sum of $k$ weighted Whitney numbers. The weighted and ranked poset $\left(S,\preceq,w\right)$ is \textit{strongly Sperner} (or has the \textit{strong Sperner property}) if it is $k$-Sperner for all $k=1,2,\cdots$.

The \textit{product} of the posets $S$ and $T$ is defined to be the Cartesian product $S\times T$, equipped with the partial order defined by 
the requirement that $(s,t)\preceq (s',t')$ in $S\times T$ if and only if $s\preceq s'$ in $S$ and $t\preceq t'$ in $T$. 
If $S$ and $T$ are weighted with weight functions $w_S$ and $w_T$, then the weight function $w_{S\times T}$ of $S\times T$ is defined to be $w_{S\times T}(s,t)=w_S(s)w_T(t)$.

\subsection{Normalized matching property}
Consider a ranked and weighted poset $\left(S,w,\preccurlyeq \right)$ with the rank function $\rho$. For any subset $A$ of $S$, 
we define the \textit{upper shade of $A$}, denoted $\nabla(A)$, as the set of all elements covering $A$. If $a'\in \nabla(A)$, 
then there exists an element $a\in A$ such that $a\preceq a'$ and $\rho(a')=\rho(a)+1$.

Let $N_r$ be the collection of all elements at rank $r$. A ranked and weighted poset $\left(S,w,\preccurlyeq \right)$ is called \textit{normal} if for any antichain $A$ subject to a subset of elements of rank $r$, the weight sum ratio of $A$ with respect to the weighted Whitney number of rank $r$ is less than or equal to the weight sum ratio of the shade of $A$ at rank $r+1$ with respect to the weighted Whitney number of rank $r+1$, i.e.,
\begin{align}\label{eq:weightsumratio}
\frac{w(A)}{w(N_i)}\leq \frac{w(\nabla (A))}{w(N_{r+1})}
\end{align}
where $A\subseteq N_i$ is an antichain, $w(A)$ is the sum of all weights of elements in $A$, and $w(N_{r})$ is the weighted Whitney number of the rank $r$. %, and $\nabla(A)$ is the collection of elements at rank $r+1$ connected to $A$ (i.e. $\nabla(A) = \{ q_{r+1}\in N_{r+1} : q_r \preccurlyeq q_{r+1} \text{for each } q_r\in A \}$).
Hsieh and Kleitman \cite{HK73} proved that the normalized matching property is preserved under the product of normal posets if it assumes log-concave weighted Whitney numbers.
\begin{prop}[See {\cite[Theorem 4.6.2]{Eng97:book} or \cite{HK73}}]\label{lemma_hsieh}
A product of two normal posets with log-concave weighted Whitney numbers is again a normal poset with log-concave weighted Whitney numbers.
\end{prop}
\begin{prop}[See {\cite[Corollary 4.5.3]{Eng97:book}}]\label{proposition_normal_strongly_spermer}
A normal poset is strongly Sperner.
\end{prop}
Applying Proposition \ref{lemma_hsieh} and Proposition \ref{proposition_normal_strongly_spermer} to chain posets, we have the following Corollary.
\begin{cor}\label{cor:stronglyspernerchain}
For each $i\in\{1,\cdots,N\}$, assume that $S(m_i)$ has log-concave weighted Whitney numbers. Then the product of chain posets $S(m_1,\cdots,m_N)$ is strongly Sperner with log-concave weighted Whitney numbers.
\end{cor}
\begin{proof}
Since the weight sum ratio in \eqref{eq:weightsumratio} of any antichain for a chain poset is always $1$, any weighted chain is normal. Thus, the conclusion follows from Proposition \ref{lemma_hsieh} and Proposition \ref{proposition_normal_strongly_spermer}.
\end{proof}
For further discussion, we need to define an order isomorphism between two posets. We say that two posets $(Q,\preccurlyeq)$ and $(R,\preccurlyeq)$ are \textit{isomorphic} if there exists a bijective map $\phi:Q\to R$ such that $q_1\preccurlyeq q_2$ iff $\phi(q_1) \preccurlyeq \phi(q_2)$ for $q_1,q_w\in Q$ and $\phi(q_1),\phi(q_2)\in R$.

\subsection{Strongly Sperner posets}

Let $N, n_1,\cdots, n_N$ be fixed positive integers. A product of chain posets $S(n_1,\cdots, n_N)$ can be defined to be a collection of $N$-tuples of integers $(a_1,\cdots,a_N)$ such that $0\leq a_i \leq n_i$ for each $i \in \{1,\cdots,N\}$. The relation $a=(a_1,\cdots, a_N) \preccurlyeq b=(b_1,\cdots, b_N)$ iff $a_i \leq b_i$ for each $i\in\{1,\cdots,N\}$. Figure 1 shows an example of the product poset $S(2,3)$. %Proposition \ref{proposition_normal_strongly_spermer} implies the following.
%\begin{cor}\label{corollary_product_chain_poset_peck}
%A product of normal chain posets $S(m_1,\cdots,m_n)$ is again normal, so strongly Sperner.
%\end{cor}

Next, we introduce the poset $M(m)$, a collection of $m$-tuples $(a_1,\cdots,a_m)$ such that $0=a_1=\cdots=a_i < a_{i+1}<\cdots<a_m\leq m$ with $i\in\{0,\cdots, m\}$. As noted, we allow one exceptional case $i=0$. If $i=0$, we mean $a_1>0$. So $0<a_1 < a_{2}<\cdots<a_m\leq m$. The relation $a=(a_1,\cdots, a_n) \preccurlyeq b=(b_1,\cdots, b_m)$ iff $a_i\leq b_i$ for all $i$. The rank $\rho(a)=\sum_{i=1}^{m} a_i$. 

Stanley \cite{Sta80} originally proved that $M(n)$ is rank-symmetric, rank-unimodal, and strongly Sperner leveraging ideas from algebraic geometry. 
After that, Proctor \cite{Pro82} gave a more accessible proof requiring the background of basic linear algebra. 
Following conventional terminology, we say that a ranked poset is \textit{Peck} if the poset is rank-symmetric, rank-unimodal, and strongly Sperner.

\begin{lem}[See \cite{Sta80,Pro82}]\label{lem:mnpeck}
	The poset $M(m)$ is a Peck poset.
\end{lem}

%Along with Lemma \ref{lem:mnpeck},  
Proctor, Saks, and Sturtevant \cite{PSS80} proved that the Peck property is invariant under the product of posets.

\begin{lem}[See {\cite[Theorem 3.2]{PSS80}}]\label{lem:productpeck}
	A product of Peck posets is again Peck, and hence strongly Sperner.
\end{lem}

	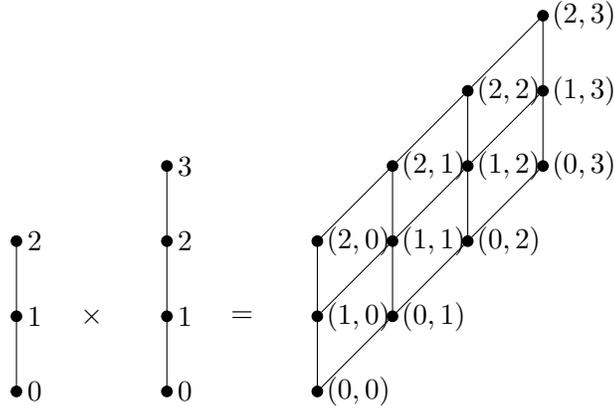
\begin{figure}
	\centering
		\begin{tikzpicture}
		    {\tikzstyle{every node}=[draw,circle,fill=black,minimum size=4pt,
		                            inner sep=0pt]
		
		    % First, draw the inner hexagon with a ``pin'' -- namely, (3214)
		    \draw (0,0) node (00) [label=right:{$(0,0)$}] {}
		        -- ++(90:1cm) node (10) [label=right:{$(1,0)$}] {}
		        -- ++(90:1cm) node (20) [label=right:{$(2,0)$}] {}
		        -- ++(45:1.414cm) node (21) [label=right:{$(2,1)$}] {}
		        -- ++(-90:1cm) node (11) [label=right:{$(1,1)$}] {}
		        -- ++(-90:1cm) node (01) [label=right:{$(0,1)$}] {}
		        -- ++(45:1.414cm) node (02) [label=right:{$(0,2)$}] {}
		        -- ++(90:1cm) node (12) [label=right:{$(1,2)$}] {}
		        -- ++(90:1cm) node (22) [label=right:{$(2,2)$}] {}
		        -- ++(45:1.414cm) node (23) [label=right:{$(2,3)$}] {}
		        -- ++(-90:1cm) node (13) [label=right:{$(1,3)$}] {}
		        -- ++(-90:1cm) node (03) [label=right:{$(0,3)$}] {};
		        
		    \draw (00) -- (01);
		    \draw (02) -- (03);
		    \draw (10) -- (11) -- (12) -- (13);
		    \draw (21) -- (22);
		    
		    \draw (-4,0) node (0) [label=right:{$0$}] {}
		    	-- ++ (90:1cm) node (1) [label=right:{$1$}] {}
		    	-- ++ (90:1cm) node (2) [label=right:{$2$}] {};
		    	
		   	\draw (-2,0) node (s0) [label=right:{$0$}] {}
		   	    	-- ++ (90:1cm) node (s1) [label=right:{$1$}] {}
		   	    	-- ++ (90:1cm) node (s2) [label=right:{$2$}] {}
		   	    	-- ++ (90:1cm) node (s3) [label=right:{$3$}] {};}
		   	    	
		   	\node at (-3,1) {$\times$};	
		   	\node at (-1,1) {$=$};
			
		\end{tikzpicture}
	\caption{Poset of $S(2,3)$}
	\end{figure}
	
	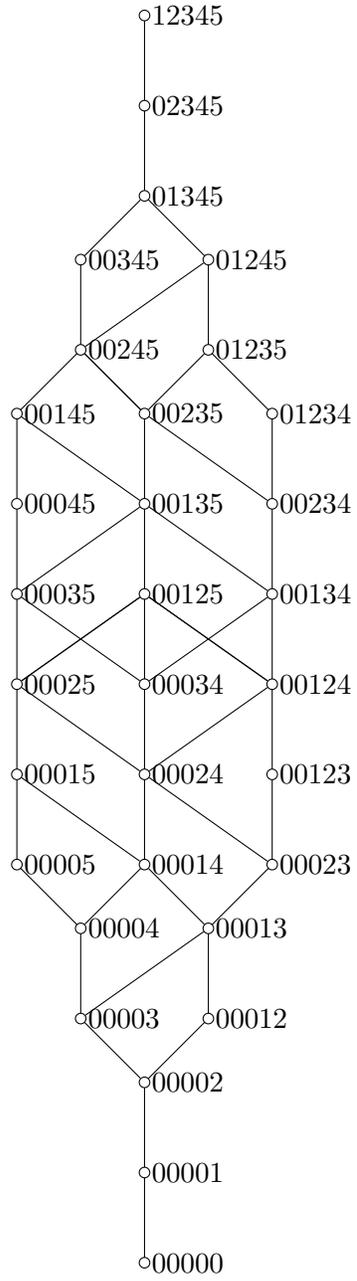
\begin{figure}
	\centering
	\begin{tikzpicture}[scale=0.8]
	    \tikzstyle{every node}=[draw,circle,fill=white,minimum size=4pt,
	                            inner sep=0pt]
	
	    % First, draw the inner hexagon with a ``pin'' -- namely, (3214)
	    \draw (0,0) node (00000) [label=right:{$00000$}] {}
	        -- ++(90:1.5cm) node (00001) [label=right:{$00001$}] {}
	        -- ++(90:1.5cm) node (00002) [label=right:{$00002$}] {}
	        -- ++(45:1.5cm) node (00012) [label=right:{$00012$}] {}
	        -- ++(90:1.5cm) node (00013) [label=right:{$00013$}] {}
	        -- ++(45:1.5cm) node (00023) [label=right:{$00023$}] {}
	        -- ++(90:1.5cm) node (00123) [label=right:{$00123$}] {}
	        -- ++(90:1.5cm) node (00124) [label=right:{$00124$}] {}
	        -- ++(90:1.5cm) node (00134) [label=right:{$00134$}] {}
	        -- ++(90:1.5cm) node (00234) [label=right:{$00234$}] {}
	        -- ++(90:1.5cm) node (01234) [label=right:{$01234$}] {}
	        -- ++(135:1.5cm) node (01235) [label=right:{$01235$}] {}
	        -- ++(90:1.5cm) node (01245) [label=right:{$01245$}] {}
	        -- ++(135:1.5cm) node (01345) [label=right:{$01345$}] {}
	        -- ++(90:1.5cm) node (02345) [label=right:{$02345$}] {}
	        -- ++(90:1.5cm) node (12345) [label=right:{$12345$}] {};
	
		\draw (00002) 
			-- ++(135:1.5cm) node (00003) [label=right:{$00003$}] {}
			-- ++(90:1.5cm) node (00004) [label=right:{$00004$}] {}
			-- ++(135:1.5cm) node (00005) [label=right:{$00005$}] {}
			-- ++(90:1.5cm) node (00015) [label=right:{$00015$}] {}
			-- ++(90:1.5cm) node (00025) [label=right:{$00025$}] {}
			-- ++(90:1.5cm) node (00035) [label=right:{$00035$}] {}
			-- ++(90:1.5cm) node (00045) [label=right:{$00045$}] {}
			-- ++(90:1.5cm) node (00145) [label=right:{$00145$}] {}
			-- ++(45:1.5cm) node (00245) [label=right:{$00245$}] {}
			-- ++(90:1.5cm) node (00345) [label=right:{$00345$}] {}
			-- (01345);
			
		\draw (00004)
			-- ++(45:1.5cm) node (00014) [label=right:{$00014$}] {}
			-- ++(90:1.5cm) node (00024) [label=right:{$00024$}] {}
			-- ++(90:1.5cm) node (00034) [label=right:{$00034$}] {}
			-- ++(90:1.5cm) node (00125) [label=right:{$00125$}] {}
			-- ++(90:1.5cm) node (00135) [label=right:{$00135$}] {}
			-- ++(90:1.5cm) node (00235) [label=right:{$00235$}] {}
			-- (00245);
			
		\draw (00003) -- (00013) -- (00014) -- (00015);
		\draw (00023) -- (00024) -- (00025) -- (00125) -- (00124);
		\draw (00024) -- (00124) -- (00125) -- (00025);
		\draw (00134) -- (00034) -- (00035) -- (00135) -- (00134);
		\draw (00135) -- (00145);
		\draw (00234) -- (00235) -- (00245) -- (01245);
		\draw (00235) -- (01235);
		
	\end{tikzpicture}
	\caption{Poset of $M(5)$}
	\end{figure}

\section{Proof of Theorem \ref{thm:main1}}\label{sec:thmproofmain1}
We establish the link between a non-negative function and a weighted chain poset. Consider a $\#$-log-concave function $f=\sum_{r=0}^{n} a_r\{x_r\}$, 
where $x_0<\cdots<x_n$ and $a_{r}^2\geq a_{r-1} a_{r+1}$. By letting $S_f:=\text{Supp}(f)=\{x_0,\cdots,x_n\}$ with a weight function $f(x_r)=a_r$, $(S_f,f,\leq)$ forms a ranked and weighted chain poset. Thus, we regard a non-negative function with finite support as a weighted chain poset:
\begin{align*}
f\equiv (S_f,f,\preccurlyeq),
\end{align*}
where the relation $\preccurlyeq$ is the same as the usual $\leq$. Since $f$ is $\#$-log-concave, weighted Whitney numbers of $S_f$ are log-concave. Similarly $f^{\#}=\sum_{r=0}^{n}a_r\{r\}$ forms a weighted chain poset $S_{f^{\#}}=\{0,\cdots,n\}$ with log-concave Whitney numbers. Based on the construction, $S_f$ is isomorphic to $S_{f^{\#}}$ by mapping $\phi(x_r)=r$, so $f(x_r)=f^{\#}(\phi(x_r))$. i.e. the isomorphism map $\phi$ can be chosen to be the rank function of $S_f$.

Next, consider $N$  non-negative functions $f_1,\cdots,f_N$, all of which are $\#$-log-concave. Define $F(x^{(1)},\cdots, x^{(N)}):=f_1 (x^{(1)}) \cdots f_N(x^{(N)})$. Similarly define $F^{\#}(x^{(1)},\cdots, x^{(N)}):=f_1^{\#} (x^{(1)}) \cdots f_N^{\#}(x^{(N)})$. As shown above, there exists an isomorphic map $\phi_i$ between $S_{f_i}$ and $S_{f_i^{\#}}$ for each $i=1,\cdots,N$. Thus, we choose $\phi:S_{f_1}\times \cdots \times S_{f_N}\to S_{f_1^{\#}}\times \cdots \times S_{f_N^{\#}}$ by
\begin{align}\label{eq:isomorphism}
\phi \left(x^{(1)}_{r_1},\cdots, x^{(N)}_{r_N} \right)=\left( \phi_1(x^{(1)}_{r_1}), \cdots, \phi_N(x^{(N)}_{r_N}) \right).
\end{align}
Then $S_{f_1}\times \cdots \times S_{f_N}$ is isomorphic to $S_{f_1^{\#}}\times \cdots \times S_{f_N^{\#}}$ by $\phi$. i.e.
\begin{align*}
\left( S_{f_1}\times \cdots \times S_{f_N} ,F ,\preccurlyeq \right) \equiv \left( S_{f_1^{\#}}\times \cdots \times S_{f_N^{\#}} , F^{\#} ,\preccurlyeq \right),
\end{align*}
where $F\left( x^{(1)}_{r_1},\cdots, x^{(N)}_{r_N} \right) = F^{\#} \left( \phi_1(x^{(1)}_{r_1}), \cdots, \phi_N(x^{(N)}_{r_N}) \right)$.
\begin{lem}\label{lem:snormal}
$S_{f_1^{\#}}\times \cdots \times S_{f_{N}^{\#}}$ forms a normal poset with log-concave weighted Whitney numbers.
\end{lem}
\begin{proof}
Each $S_{f_i}^{\#}$ is a chain, thus it is a normal poset with log-concave weights. Corollary \ref{cor:stronglyspernerchain} confirms that the product of normal posets is again normal with log-concave weighted Whitney numbers.
\end{proof}

Next, we establish a link between a product of posets and a convolution of non-negative functions through an antichain. We define a level set
\begin{align*}
L[x] := \left\{ \left( x^{(1)},\cdots,x^{(N)} \right) : x^{(1)}+\cdots + x^{(N)}=x, x^{(i)}\in S_{f_i} \text{ for $i=1,\cdots,N$} \right\}.
\end{align*}
\begin{lem}\label{lem:antichain}
$\phi\left( L[x] \right)$ forms an antichain in $S_{f_1^{\#}}\times \cdots \times S_{f_N^{\#}}$.
\end{lem}
\begin{proof}
Note that $\phi$ in \eqref{eq:isomorphism} is a bijective and order-preserving map. Thus, it suffices to consider elements in $L[x]$. Suppose that there exist two distinct comparable elements $\mathbf{x}:=\left( x^{(1)},\cdots,x^{(N)} \right)$ and $\mathbf{y}:= \left( y^{(1)},\cdots,y^{(N)} \right)$ such that $\mathbf{x},\mathbf{y}\in L[x]$ and $\textbf{x} \preccurlyeq \mathbf{y}$. This implies $x^{(i)}\leq y^{(i)}$ for each $i=1,\cdots,N$. Since $\mathbf{x}$ and $\mathbf{y}$ are distinct, there exists some $j$ such that $x^{(j)} < y^{(j)}$. Hence
\begin{align*}
x^{(1)} + \cdots + x^{(N)} < y^{(1)} + \cdots + y^{(N)}.
\end{align*}
This contradicts the fact that both $\mathbf{x}$ and $\mathbf{y}$ are in $L[x]$.
\end{proof}
Since $S_{f_1^{\#}}\times \cdots \times S_{f_N^{\#}}$ is strongly Sperner, majorization follows.
\begin{prop}\label{prop:hashlogconcave}
If $f_1,\cdots,f_N$ are $\#$-log-concave probability mass functions,
\begin{align*}
f_1\star \cdots \star f_N \prec f_1^{\#} \star \cdots \star f_{N}^{\#}.
\end{align*}
\end{prop}
\begin{proof}
Let $f_L:=f_1\star \cdots \star f_N$ and $f^{\#}_R := f_1^{\#} \star \cdots \star f_{N}^{\#}$. Convolutions in $f_L$ can be fully represented by
\begin{align*}
f_L(x) = \sum_{(x^{(1)},\cdots, x^{(N)})\in L[x]} f_1(x^{(1)})\cdots f_N(x^{(N)}).
\end{align*}
The isomorphism $\rho$ in \eqref{eq:isomorphism} and Lemma \ref{lem:antichain} imply that $f_L(x)$ can be regarded as a sum of weights of an antichain in $ S_{f_1^{\#}}\times \cdots \times S_{f_N^{\#}}$. Since $ S_{f_1^{\#}}\times \cdots \times S_{f_N^{\#}}$ is a normal poset by Lemma \ref{lem:snormal}, $S_{f_1^{\#}}\times \cdots \times S_{f_N^{\#}}$ is strongly Sperner. Therefore
\begin{align*}
\sum_{i=1}^{k} f_L^{[i]} \leq \sum_{i=1}^{k} f_R^{[i]},
\end{align*}
where the left-hand side corresponds to the sum of $k$ antichains and the right-hand side corresponds to the sum of $k$ largest Whitney numbers in  $S_{f_1^{\#}}\times \cdots \times S_{f_N^{\#}}$. Since $f_L$ and $f_R$ are still probability mass functions,
\begin{align*}
\sum_{i=1}^{M} f_L^{[i]} = \sum_{i=1}^{M} f_R^{[i]}=1,
\end{align*}
for some sufficiently large $M>0$. Thus majorization follows.
\end{proof}
We finally note that Theorem \ref{thm:main1} is a restatement of Proposition \ref{prop:hashlogconcave} using the notions of random variables.

\section{Proof of Theorem \ref{thm:main2}}\label{sec:thmproofmain2}

We assume that $0<a_1<\cdots<a_N$. Let $X_{i,j}$ for $1\leq i \leq N$ and $1\leq j \leq m_i\leq N$ be independent random variables following $\text{Bernoulli}\left(\frac{1}{2}\right)$. From the assumption, $1\leq m_N\leq m_{N-1}\leq \cdots \leq m_1$. Then, we can decompose each $Y_i$ as follows:
\begin{align*}
Y_1 &:= X_{1,1} + \cdots + X_{1,m_1},\\
\vdots&\qquad\qquad\quad\vdots\\
Y_N &:= X_{N,1} + \cdots + X_{N,m_N}.
\end{align*}
By the construction, $Y_i$'s are independent random variables following $\text{Binomial}\left(m_i,\frac{1}{2} \right)$ for all $1\leq m_N\leq m_{N-1}\leq \cdots \leq m_1$. We denote $n_j$ by the number of defined $X_{i,j}$'s for each $1\leq j\leq N$. Let $\mathbf{Z_j} := \left( X_{1,j}, \cdots, X_{m_j,j}\right)$.

For each $j$, consider an element $\mathbf{i_j}=(b_1,\cdots,b_{m_j})$ in $M(m_j)$. We encode $b_k=i>0$ for some $k$ in $\mathbf{i_j}$ iff $X_{i,j}=1$. Otherwise, $b_k=0$. For example, when $m_j=5$,
\begin{align*}
\mathbf{i_j}=(0,0,0,2,4) \quad\text{iff} \quad \mathbf{Z_j} = \left( X_{1,j}=0, X_{2,j}=1,X_{3,j}=0,X_{4,j}=1,X_{5,j}=0\right).
\end{align*}
Thus as described above, there exists a bijective link between an element $\mathbf{i_j}$ in $M(m_j)$ and each realization of the random vector $\mathbf{Z_j}$. Furthermore, we are able to construct another bijective map between an element $(\mathbf{i_1},\cdots,\mathbf{i_N})$ in $M(m_1)\times \cdots \times M(m_N)$ and each realization of the random array $(\mathbf{Z_1},\cdots, \mathbf{Z_N})$.

Let $L_j(\mathbf{Z_j}) := a_1X_{1,j} + \cdots + a_{j} X_{m_j,j}$ for each $1\leq j\leq N$. Based on the construction, we see that
\begin{align*}
Y_L:&= a_1Y_1 + \cdots + a_N Y_N = L_1(\mathbf{Z_1}) + \cdots + L_N(\mathbf{Z_N}).
\end{align*}
As demonstrated in Section \ref{sec:thmproofmain1}, we similarly define a level set $L_j[x_j]$ in $M(m_j)$ as follows:
\begin{align*}
L_j[x_j] := \left\{ \mathbf{i_j}\in M(m_j) : \mathbf{i_j} \text{ bijectively corresponds to } \mathbf{Z_j} \text{ such that } L_j(\mathbf{Z_j})=x_j \right\}.
\end{align*}
\begin{lem}\label{lem:mnantichain}
$L_j[x_j]$ forms an antichain in $M(m_j)$.
\end{lem}
\begin{proof}
Suppose that there exist two distinct elements $\mathbf{i_j}$ and $\mathbf{i'_j}$ in $L_j[x_j]$ such that $\mathbf{i_j} \preccurlyeq \mathbf{i'_j}$. Assume that  $\mathbf{i_j}$ and $\mathbf{i'_j}$ correspond to $\mathbf{Z_j}$ and $\mathbf{Z'_j}$, respectively. Since $\mathbf{i_j}=(b_1,\cdots,b_{m_j})$ and $\mathbf{i'_j}=(b'_1,\cdots,b'_{m_j})$ are distinct, there exists some $k>0$ such that $b_k < b'_k$. Since $a_i>0$, $0<a_{b_k}< a_{b'_k}$. It implies that
\begin{align*}
L_j(\mathbf{Z_j}) < L_j(\mathbf{Z'_j}).
\end{align*}
This contradicts that both $\mathbf{i_j}$ and $\mathbf{i'_j}$ are in $L_j[x_j]$.
\end{proof}
More generally, we define a level set $L[x]$ in $M(m_1)\times \cdots \times M(m_N)$ as
\begin{align*}
L[x] &:= \{ (\mathbf{i_1},\cdots,\mathbf{i_N})\in M(m_1)\times \cdots \times M(m_N) : (\mathbf{i_1},\cdots,\mathbf{i_N}) \\
& \text{ bijectively corresponds to } (\mathbf{Z_1},\cdots,\mathbf{Z_N}) \text{ such that } L_1(\mathbf{Z_1}) + \cdots + L_N(\mathbf{Z_N})=x \}.
\end{align*}
Following the same argument in Lemma \ref{lem:mnantichain}, we see that $L[x]$ forms an antichain. We omit the proof for simplicity.
\begin{lem}
$L[x]$ forms an antichain in $M(m_1)\times \cdots \times M(m_N)$.
\end{lem}
Next, Lemma \ref{lem:productpeck} implies that a product of posets $M(m_1)\times \cdots \times M(m_N)$ is again Peck, so it is strongly Sperner.
\begin{lem}\label{lem:productmnpeck}
$M(m_1)\times \cdots \times M(m_N)$ is Peck, thus strongly Sperner.
\end{lem}
We note that $|M(m_j)|=2^{m_j}$, so $|M(m_1)\times \cdots \times M(m_N)|=2^{m_1+\cdots+m_N}$. Then,
\begin{align*}
\mathbf{P} \left( Y_L =x \right) = \frac{\left| L[x] \right|}{2^{m_1+\cdots+m_N}}.
\end{align*}
Before finding the link between strong Sperner property and majorization, it is necessary to identify the saturated case meaning that it achieves the maximal sum. When $a_1=1,\cdots,a_N=N$, let
\begin{align*}
Y_R:&= Y_1 + \cdots + N Y_N = R_1(\mathbf{Z_1}) + \cdots + R_N(\mathbf{Z_N}),
\end{align*}
where $R_j(\mathbf{Z_j}) := X_{1,j} + \cdots + m_j X_{m_j,j}$ for each $1\leq j\leq N$. Then as Stanley and Proctor explained in \cite{Sta80, Pro82}, we see that the size of each level set of $R_j(\mathbf{Z_j})$ has a bijective correspondence to a Whitney number of $M(m_j)$ by matching the rank to the value $R_j(\mathbf{Z_j})$. Hence each level set of $Y_R$ has a bijective correspondence to a Whitney number of $M(m_1)\times \cdots \times M(m_N)$ by applying the property of the product of posets. Therefore we confirm that $Y_R$ is the saturated case.

Now it remains to establish majorization through the strong Sperner property of $M(m_1)\times \cdots \times M(m_N)$. Let $\mathbb{Z}_+^0$ be a collection of non-negative integers. Observe that
\begin{align}\label{eq:majorization_mn}
\left( 2^{m_1+\cdots+m_N} \right)\sup_{C\subset \mathbb{Z}_+^0 } \sum_{|C|=k} \mathbf{P}\left( Y_L \in C \right) \leq \left( 2^{m_1+\cdots+m_N} \right) \sup_{C\subset \mathbb{Z}_+^0} \sum_{|C|=k} \mathbf{P}\left( Y_R \in C \right),
\end{align}
where the left-hand side corresponds to the sum of weights from $k$ antichains in $M(n_1)\times \cdots \times M(n_N)$ and the right-hand side exactly corresponds to the sum of $k$-largest Whitney numbers from $M(m_1)\times \cdots \times M(m_N)$. We see that the equation \eqref{eq:majorization_mn} is confirming the condition \eqref{eqn:majorization_condition_1}. The condition \eqref{eqn:majorization_condition_2} follows using the fact that the total sum equals $1$ as probability mass functions. Thus, the conclusion follows by re-stating it using the notions of random variables.

%Finally, we discuss the general case $0< a_1 \neq \cdots \neq a_N$. In this case, we can define a permutation $\sigma$ such that they can be re-ordered as $0<a_{\sigma(1)}<\cdots<a_{\sigma(N)}$. By matching indices $\{\sigma(1),\cdots,\sigma(N)\}$ to $\{1,\cdots,N\}$ in $X_{i,j}$, we can apply the same arguments as above. Even though we reorder indices, the saturated case still remains the same. Therefore the proof follows for the general case as well.

%\section{Conclusion}
%In this paper, we prove that the strong Sperner poset can be translated to majorization in the sense of function ordering. Then we prove R\'enyi entropy inequalities including discrete entropy power inequalities over the integers and the integer lattice. Furthermore, our results extend the classical result of Erd\H{o}s-Moser problem and Hardy-Littlewood-P\'olya's rearrangement inequality.

\section*{Acknowledgments}
The authors are indebted to Manjunath Krishnapur for very useful discussions, and to two anonymous reviewers for detailed
and constructive comments that greatly improved the exposition.
This work was supported in part by the U.S. National Science Foundation through grants DMS-1409504 (CAREER) and CCF-1346564.

%\section*{References}
\bibliographystyle{plain}
%% `Elsevier LaTeX' style
%\bibliographystyle{elsarticle-num}
%\bibliography{$HOME/Dropbox/MAIN/WRITINGS/CommonResources/pustak}

\end{document}